\documentclass[10pt]{article}
\usepackage{amssymb,amsfonts}
\usepackage{amsmath,amsthm,amsfonts,amssymb}
\usepackage[usenames]{color}
\usepackage{hyperref}

\newcommand{\be}{\begin{enumerate}}
\newcommand{\ee}{\end{enumerate}}
\newcommand{\beq}{\begin{equation}}
\newcommand{\eeq}{\end{equation}}

\newtheorem{proposition}{Proposition}
\newtheorem{theorem}{Theorem}
\newtheorem{lemma}{Lemma}
\newtheorem{definition}{Definition}
\newtheorem{remark}{Remark}

\title{A Whitehead algorithm for toral relatively hyperbolic groups}
\author{Olga Kharlampovich~\footnote{Department of Mathematics and Statistics, Hunter College City University NY, New York; \newline e-mail: okharlampovich@gmail.com.} and Enric
Ventura~\footnote{Departament Matem\`{a}tica  Aplicada III, Universitat Polit\`{e}cnica de Catalunya, Manresa, Barcelona,
CATALUNYA; e-mail: enric.ventura@upc.edu.}}
\begin{document}

\maketitle

Consider the following problem about a group $G$ and its automorphisms. Given finite tuples $(u_1,\ldots ,u_n)$ and
$(v_1,\ldots ,v_n)$ of elements of $G$, decide whether there is an automorphism of $G$ taking $u_i$ to $v_i$ for all
$i$. If so, find one. We will call this problem the Whitehead Problem (WhP) for $G$. The generalized WhP is solvable if
there is an algorithm that given finite tuples $(u_{11},\ldots ,u_{1n_1},\ldots ,u_{k1},\ldots ,u_{kn_k})$ and
$(v_{11},\ldots ,v_{1n_1},\ldots ,v_{k1},\ldots ,v_{kn_k})$ of elements of $G$, decides whether there is an
automorphism of $G$ taking $u_{ij}$ to $v_{ij}^{g_i}$ for all $i,j$, and some $g_i\in G$.

Whitehead found an algorithm solving the WhP for a finitely generated free group~\cite{W}. Whitehead problem was also
solved for surface groups in~\cite{LV} and, recently, for hyperbolic groups~\cite{DGu}. The generalized WhP has been
also solved for torsion free hyperbolic groups~\cite{BV}.

Let $\mathcal G$ denote the class of toral relatively hyperbolic groups (torsion-free relatively hyperbolic groups with
abelian parabolic subgroups).

\begin{theorem}
The  WhP is solvable for  $G\in{\mathcal G}$.
\end{theorem}

This result implies that the WhP is solvable in limit groups and torsion-free hyperbolic groups. Notice, that in
\cite{DGu} the WhP was solved as a consequence of the solution of the isomorphism problem for hyperbolic groups. For
every two tuples of elements $(u_1,\ldots ,u_n)$ and $(v_1,\ldots ,v_n)$ of $G$ one has to construct JSJ decompositions
of auxiliary groups $G_1(u_1,\ldots ,u_n)$ and $G_2(v_1,\ldots ,v_n)$ obtained from $G$ and then decide whether $G_1$
and $G_2$ are isomorphic. The algorithm in~\cite{DGu} for the construction of the JSJ decomposition involves complete
enumeration of all presentations of the group obtained by Tietze transformations. The advantage of our approach is that
we have to construct the JSJ decomposition only for the original $G$ and only once, therefore it has lower complexity.

We begin the proof by mentioning that in 1984, Collins and Zieschang extended Whitehead's methods to free products of
finitely many freely indecomposable groups, assuming that WhP can be solved in each factor~\cite{CZ1},\cite{CZ2}.
Therefore, without loss of generality, we can restrict ourselves to consider only the case when $G$ is freely
indecomposable. Groups from this class have algorithmically computable canonical $Out(G)$–-invariant abelian JSJ
decompositions~\cite{DG} with all parabolic subgroups being elliptic.

\begin{definition}\label{defn:canonical autos1}
Let $G=A*_{C}B$ be an elementary abelian splitting of a freely indecomposable group $G$. For $c\in C$ we  define an
automorphism $\phi_c :G\rightarrow G$ such that $\phi_c(a)=a$ for $a\in A$ and $\phi_c(b)=b^{c}=c^{-1}bc$ for  $b\in
B$.

If $G=A*_{C}=\langle A,t|c^{t}=c', c\in C\rangle$ (where $c$ and $c'$ represent the images of the same element of $C$
under the two given inclusions $\alpha,\, \omega \colon C\to A$) then for $c \in C$ define $\phi_c :G\rightarrow G$
such that $\phi_c (a)=a$ for $a\in A$ and $\phi_c (t)=ct$.

In both cases, we call $\phi_c$ a {\em Dehn twist} obtained from the corresponding elementary abelian splitting of $G$.

Note that, if $G=A*_{C}B$, then every
automorphism of $B$ acting trivially on $C$ can be extended to a unique automorphism of $G$ acting trivially on $A$.

\end{definition}

\begin{definition}\label{defn:canonical autos}
Let $G$ be a freely indecomposable group, and let $\Gamma(V,E;T)$  be an Abelian JSJ decomposition of $G$ (computable
from a given presentation for $G$).  We define the group  $Out_{\Gamma}(G)$, to be the subgroup of $Out(G)$ generated
by the following types of automorphisms of $G$:
\begin{enumerate}
\item \label{e:Dehn twist} Dehn twists along edges in $\Gamma$,
\item \label{e:abelian vertices} automorphisms of an abelian vertex group that preserve the peripheral subgroups of the
group,
\item \label{e:mapping class groups} automorphisms of a $QH$-vertex group $G_u$ preserving the peripheral subgroups of
the group, up to conjugacy (geometrically, these are Dehn twists along simple closed curves on the punctured surface
$\Sigma$ with $\pi_1(\Sigma)\cong G_u$).
\end{enumerate}
The full preimage of $Out_{\Gamma}(G)\leqslant Out(G)$ in $Aut(G)$ (which, of course, contains all inner automorphisms)
is called the \emph{group of canonical automorphisms with respect to $\Gamma$}, denoted $AutC_{\Gamma}(G)$.
\end{definition}

\begin{lemma}~\cite{GL1}
With the notation of Definition~\ref{defn:canonical autos}, $[Out(G):Out_{\Gamma}(G)]<\infty$ and hence, the group of
canonical automorphisms of $G$ has finite index in the group of all automorphisms of $G$,
$[Aut(G):AutC_{\Gamma}(G)]<\infty$.
\end{lemma}

The following proposition implies that one can effectively find representatives of all conjugacy classes of
automorphisms of rigid subgroups compatible with edge groups.

\begin{proposition} \label{DG}[Theorem 5.11, \cite{DG}]
Let $G$ (respectively, $H$) be a toral relatively hyperbolic group, and let $\mathcal A=(A_1,\ldots ,A_n)$ (resp.,
$\mathcal B=(B_1,\ldots ,B_n)$) be a finite list of non-conjugated maximal abelian subgroups of $G$ (resp., $H$) such
that the abelian decomposition of $G$ modulo $\mathcal A$ (resp. of $H$ modulo $\mathcal B$) is trivial. The number of
conjugacy classes of monomorphisms from $G$ to $H$ that map subgroups from $\mathcal A$ onto conjugates of the
corresponding subgroups from $\mathcal B$ is finite. A set of representatives of the equivalence classes can be
effectively found.

If $G=H$, then there is at most a finite number of conjugacy classes of automorphisms compatible with the peripheral
structure, and there is an algorithm to find representatives of all of them.
\end{proposition}

We can suppose that $G$ in not abelian and not a closed surface group. We compute a canonical JSJ decomposition
$\mathcal D$ for $G$, with the extra property that parabolic subgroups are elliptic. Notice that $AutC_{\mathcal D}(G)$
consists of automorphisms $\phi$ that map every vertex group of $\mathcal D$ into a conjugate of itself and have the
following property: for any rigid subgroup $H$ there exists $g\in G$ such that $\phi (h)=ghg^{-1}$ for any $h\in H$. By
Lemma \ref{DG}, $[Aut (G):AutC_{\mathcal D}(G)]<\infty.$ Every automorphism of $G$ maps $H$ to a conjugate of a rigid
subgroup, and there is only a finite number (up to conjugation) of automorphisms of a rigid subgroup onto itself
preserving its peripheral subgroups up to conjugacy. We can effectively find all such automorphisms and, therefore,
compute left coset representatives $\tau_1,\ldots ,\tau _k$ of $AutC_{\mathcal D}(G)$ in $Aut(G)$. Then, to decide
whether the tuple $(u_1,\ldots ,u_n)$ is in the orbit of the tuple $(v_1,\ldots ,v_n)$ with respect to $Aut(G)$, we
have to decide whether $(\tau _i(u_1),\ldots ,\tau _i(u_n))$ is in the orbit of $(v_1,\ldots ,v_n)$ with respect to
$AutC_{\mathcal D}(G)$, for some $\tau_i$. Therefore, to solve the WhP in $G$ we are reduced to solving the WhP for the
group of canonical automorphisms $AutC(G)$.

Combining foldings and slidings, we can transform the JSJ decomposition $\mathcal D$ in such a way that each non-cyclic
abelian vertex group that is connected to a rigid subgroup is connected to only one vertex group and this vertex group
is rigid. We fix such a decomposition and denote it again by $\mathcal D$. We also fix a maximal forest $T_1$ joining
all non-abelian vertex groups, and a maximal subtree $T$ of $\mathcal D$ with $T_1\subseteq T$. From now on, all
canonical automorphisms will be with respect to $\mathcal D$. We order edges in $T_1$ and take free products with
amalgamation following this order; then, we order the rest of the edges of $\mathcal D$ that are not in $T$, assign
stable letters to these edges and take HNN extensions in this order. After that, we order edges in $T-T_1$.

\begin{lemma}\label{QH1}
Let $C=<c>, D=<d>$, $C\neq D$ be edge groups of a QH-subgroup $Q$. For any $u,v\in Q$ there exists a bound on possible
numbers $m,n$ such that there exists an automorphism $\alpha$ of $Q$ with $\alpha(u)=d^{m\delta}vc^{n\gamma}, \alpha
(c)=c^{\gamma}, \alpha (d)=d^{\delta}$, for some $\gamma ,\delta \in Q$. Moreover, there exists an algorithm to find
such a bound, all valid values of $m,n$ and, for each pair $m,n$, an automorphism $\alpha$.
\end{lemma}

\begin{proof}
The proof is similar to the proof of Lemmas~3.4 and~3.5 in~\cite{LV}. Notice that, under the assumptions $\alpha
(d)=d^{\delta}$ and $\alpha (c)=c^{\gamma}$, $\alpha (u)=d^{m\delta}vc^{n\gamma}$ iff $\alpha (d^{-m}uc^{-n})=v$. We
choose a base point $P$ on the boundary corresponding to $C$ and represent $u,v$ as closed curves on the surface.
Moreover, we take minimal representatives in the sense of \cite{LV}. Then  minimal representatives for $v$ and
$d^{-m}uc^{-n}$ must have the same number of self-intersection points. The existence of such $\alpha$ can be
effectively verified as in~\cite{LV}.
\end{proof}

\begin{remark}
For each $u,v,\gamma \in Q$ there is at most one number $n$ for which there exists $\alpha$ with the properties that
$\alpha (u)=vc^{n\gamma}$ and $\alpha (c)=c^{\gamma}$. Indeed if, in addition, $\beta (u)=vc^{m\gamma}$ and $\beta
(c)=c^{\gamma}$ for some others $m$ and $\beta$, then $\beta\alpha^{-1} (vc^{n\gamma})=vc^{m\gamma}$ and $\beta\alpha
^{-1}(c^{\gamma})=c^{\gamma}$. Now, choosing the basepoint $P$ on $c^{\gamma}$, the curves $vc^{n\gamma}$ and
$vc^{m\gamma}$ have different number of self-intersections unless $n=m$.
\end{remark}

A multiple version of Lemma~\ref{QH1} gives the following lemma.

\begin{lemma}\label{QH2}
Let $C=<c>, D=<d>$, $C\neq D$ be edge groups of a QH-subgroup $Q$. For any finite set $I$ and tuples of elements
$(u_i)_{i\in I}$ and $(v_i)_{i\in I}$ from $Q$, there exists a bound on possible numbers $m_i,n_i$ for which there
exists an automorphism $\alpha$ of $Q$ with $\alpha(u_i)=d^{m_i\delta}v_ic^{n_i\gamma}, \alpha (c)=c^{\gamma}, \alpha
(d)=d^{\delta}$, for some $\gamma ,\delta \in Q$. Moreover, there exists an algorithm to compute such bound, all valid
values of $m_i,n_i$ and, for each pair of tuples $(m_i)_{i\in I}$ and $(n_i)_{i\in I}$, an automorphism $\alpha$.
\end{lemma}

\begin{lemma}
Let $G\in{\mathcal G}$, and take elements $v,w\in G$ and an abelian subgroup $C\leqslant G$. If either $v$ or $w$ do
not belong to the maximal abelian subgroup containing $C$, then there exists at most one pair of elements $\gamma_1,
\gamma _2\in C$ such that $w=\gamma_1v\gamma _2$; furthermore, there is an algorithm deciding whether it exists or not
and, in the affirmative case, computing such elements $\gamma_1, \gamma _2\in C$.
\end{lemma}

\begin{proof}
Assume $v$ (or $w$) does not belong to the maximal abelian subgroup $C'$ of $G$ containing $C$, and suppose
$w=\gamma_1v\gamma _2 =\gamma_3v\gamma _4$, for some $\gamma_1, \gamma_2, \gamma_3, \gamma_4 \in C$. Then,
$v^{-1}(\gamma _3^{-1}\gamma_1)v=\gamma _4\gamma_2^{-1}.$ By the CSA property of toral relatively hyperbolic groups
(see Lemma~2.5 in~\cite{Gro}), $C'$ is malnormal and so, $\gamma_1=\gamma_3$ and $\gamma_2=\gamma_4$.

To make the decision algorithmic, we remind that equations are solvable in toral relatively hyperbolic
groups~\cite{Dah}. Therefore, we can decide whether such $\gamma_1, \gamma _2$ exist or not in $G$ (the fact that
$\gamma _i\in C$ can be expressed by the equation $[c,\gamma _i]=1$).
\end{proof}

\begin{definition}
Let ${\mathcal D}$ be an abelian JSJ decomposition of a freely indecomposable $G\in{\mathcal G}$ with a graph $\Gamma$
that does not have abelian vertices. Let $\Gamma _1$ be a connected subgraph of $\Gamma$, and $B$ be the fundamental
group of $\Gamma _1$, $B\leqslant G$. An automorphism of $B$  is called ${\mathcal D}$-compatible if it  takes vertex
subgroups of $\Gamma _1$ into conjugates of themselves, and edge subgroups of these vertices into conjugates of edge
subgroups. Let $C=G_{e}$ be an edge group in ${\mathcal D}$, $e\not\in\Gamma _1$, $K=G_{e'}$ be different edge group,
$e'\not\in \Gamma _1$, and suppose that for any $u,v\in B$ there exists only finitely many elements $c\in C$ and $k\in
K$ such that $u$ is taken to $k^{\delta} vc^\gamma$ by a ${\mathcal D}$-compatible automorphism of $B$ sending $c$ to
$c^{\gamma}$ and $k$ to $k^{\delta}$. We say that the \emph{special Whitehead problem} with respect to $K, C$ is
solvable if $K, C$ satisfy this property and for any $u,v\in B$ there is an algorithm to decide whether there exist
$\gamma, \delta\in B, k\in K, c\in C$ such that  $u$ is taken to $k^{\delta} vc^\gamma$ by a ${\mathcal D}$-compatible
automorphism of $B$ sending $c$ to $c^{\gamma}$ and $k$ to $k^{\delta}$ and to find all such $\gamma, \delta, k, c$ and
the corresponding automorphism. If, instead of $u,v$, the same is true for tuples of elements $(u_1,\ldots ,u_m)$ and
$(v_1,\ldots ,v_m)$, we say that the special Whitehead problem with respect to $K, C$ (SWhP(K,C)) is solvable for
tuples.
\end{definition}

\begin{lemma}\label{2}
Let ${\mathcal D}$ be an abelian  JSJ decomposition of a freely indecomposable $G$ without abelian vertex groups.
Suppose $B$ (as in Definition 3) has solvable $SWhP (K, C)$ for tuples, where $C=G_e, K=G_{e'}$ are  edge groups of
$\Gamma$, and $D$ is the vertex group in the abelian decomposition $\mathcal D$ corresponding to the other endpoint of
$e$ not in $\Gamma _1$ ($D$ is then either a MQH subgroup $Q$ or a rigid subgroup $R$). Then $B\ast _{C} D$ has
solvable $SWhP(K,C_1)$ for tuples, for any edge group $C_1=G_{e_1}$ ($e\neq e_1$) of ${\mathcal D}$ belonging to $D$.
\end{lemma}

\begin{proof}
It is enough to prove the lemma for the subgroup of canonical automorphisms fixing $D$. Denote it by $AutC_D(G)$. Let
$\alpha \in AutC_D(G)$. The restrictions of $\alpha$ to all QH-subgroups are automorphisms that map edge subgroups into
their conjugates. If $D=R$ is a rigid subgroup, then the statement follows from Lemma 4 because $\alpha$ acts trivially
on $D$.

If $D$ is a QH subgroup, then we can assume that $\alpha$ maps it to itself, and maps $C$ to itself element-wise.

Since $\alpha$ is not a conjugation on $B$, $e'\neq e.$ Suppose, first that $u,v\in D$. Let $c_1 \in C_1$, suppose that
$\alpha$ is a ${\mathcal D}$-compatible automorphism of $D$ such that $\alpha (u)=v{c_1}^{n\gamma}$, $\alpha
(c_1)={c_1}^{\gamma }$. There is only a finite number of possible such $n$. It follows from Lemma 3.5 \cite{LV} that
$c_1^{n\gamma}$ can  be effectively found. By Lemma 3.4 \cite{LV}, SWhP$(C_1)$ is solvable in $D$ for tuples. For any
two tuples $(u_1,\ldots ,u_m)$ and $(v_1,\ldots ,v_m)$, there are finitely many combinations ${c_1}^{n_1},\ldots
,{c_m}^{n_m}\in C_1$ such that   $u_1,\ldots ,u_m$  can be taken to $v_1c^{n_1\gamma},\ldots ,v_mc^{n_m\gamma}$.

Let now, $u=b_1d_1\cdots b_nd_n$ and $v=\bar b_1\bar d_1\cdots \bar b_n\bar d_n$, where $b_i,\bar b_i\in B,$ $d_i,\bar
d_i\in D$ be normal forms of $u$ and $v$ in $B*_CD$.

We assume, first that $C$ and $C_1$ are not conjugate in $D$. Without loss of generality, we can assume that $u$ and
$v$ are cyclically reduced. Every $\mathcal D$-compatible automorphism (of $B*_CD$) $\alpha$ taking $u$ to
$k^{\delta}v{\gamma} ^{\sigma}$, $k\in K, $ should act as follows:
$$
\begin{array}{ll} \alpha (b_1)=k^{\delta}\bar b_1c^{k_1}, & \\ \alpha (d_i)=c^{-k_{i}}\bar d_{i} c^{m_i} &
\text{for } i=1,\ldots ,n-1, \\ \alpha (b_i)=c^{-m_{i-1}}\bar b_ic^{k_i} & \text{for } i=2,\ldots ,n, \\ \alpha
(d_n)=c^{-k_n}\bar d_n{\gamma}^{\sigma}. & \end{array}
 $$
Moreover, the number of possible values for $\gamma , k,\ k_1, k_n$ is finite. Therefore the number of possible values
for $k_i,m_i$ is finite by Lemma 2.5 \cite{LV}. Since $SWhP(K,C)$ is solvable for tuples in $B$ and $SWhP(C,C_1)$ is
solvable for tuples in $D$ (Lemma 3), one can decide whether some $\mathcal D$-compatible automorphisms $\alpha $ of
$B$ and $\beta$ of $D$ and a tuple of integers $(k_1,m_1,\ldots ,k_n,m_n)$  exist such that
$$
\begin{array}{ll} \alpha (b_1)=k^{\delta}\bar b_1c^{k_1}, & \\ \beta (d_i)=c^{-k_{i}}\bar d_{i} c^{m_i} &
\text{for } i=1,\ldots ,n-1, \\ \alpha (b_i)=c^{-m_{i-1}}\bar b_ic^{k_i} & \text{for } i=2,\ldots ,n, \\ \beta
(d_n)=c^{-k_n}\bar d_n{\gamma}^{\sigma}. & \end{array}
 $$

If they exist then there also exists a $\mathcal D$-compatible automorphism (of $B*_CD$) fixing $C_1$ and taking $u$ to
$k^{\delta}v\gamma ^{\sigma}$, $k\in K,\gamma\in C_1$; otherwise, it does not exist.

Now we consider the case when $C$ and $C_1$ are conjugate in $D$. In this case, we can assume $C=C_1$. Hence, every
$\mathcal D$-compatible automorphism $\alpha$ (of $B*_CD$) fixing $C_1$ and taking $u$ to $k^{\delta}v\gamma
^{\sigma}$, $\gamma\in C_1$ should act as follows:
 $$
\begin{array}{ll} \alpha (b_1)=\gamma _0k^{\delta}\bar b_1c^{k_1}, & \\ \alpha (d_i)=c^{-k_i}\bar d_{i} c^{m_i} &
\text{for } i=1,\ldots ,n-1, \\ \alpha (b_i)= c^{-m_{i-1}}\bar b_ic^{k_i} & \text{for } i=2,\ldots ,n, \\ \alpha
(d_n)=c^{-k_n}\bar d_n\gamma^{\sigma}, \\ \end{array}
 $$
where $\gamma_0, \gamma, c\in C$, $i=1,\ldots ,n$.

Since we can post-compose the restriction of $\alpha$ on $B$ and on $D$ with conjugation by $\gamma _0$, the question
of finding such $\alpha$ is equivalent to the problem of finding $\alpha$ when $\gamma _0=1.$ And that problem have
been considered in the previous case.

The lemma is proved for elements. The same proof works similarly if, instead of $u$ and $v$, we consider tuples of
elements.
\end{proof}

\begin{lemma}~\label{3}
Suppose $B$ has solvable $SWhP(K_1,C_1)$ for tuples for any edge groups $K_1$, $C_1=G_{e_1}$ of $\Gamma _1$, and $D$ is
a (non-abelian) vertex group in the abelian decomposition $\mathcal D$ not in $\Gamma _1$. Then $B\ast _{C} D$ has
solvable $SWhP(K_2,C_2)$ for tuples, for any edge groups $K_2$, $C_2=G_{e_2}$ of ${\mathcal D}$.
\end{lemma}

\begin{proof}
If $e_2$ is an edge outgoing from the vertex with vertex group $D$, then the statement follows from the previous lemma.
If $e_2$ is an edge outgoing from the vertex with vertex group in $\Gamma _1$, then $C_2$ is an edge group of $\Gamma
_1$. Then we can use the fact that $B$ has solvable $SWhP (C_1)$ for tuples for any edge group $C_1=G_{e_1}$ of $\Gamma
_1$, in particular for $C_2$, and write a proof similar to the proof of the previous lemma with $u=b_1d_1\cdots b_n$
and $v=\bar b_1\bar d_1\cdots \bar b_n$.
\end{proof}

\begin{lemma}\label{base} Let $B$ be the fundamental group of a connected subgraph $\Gamma _1$ of $\Gamma$.
$SWhP(K,C)$ is solvable for tuples in $B$ for any edge groups $K,C$ of $\Gamma$.
\end{lemma}

\begin{proof}
We use Lemma \ref{3} and add to $\Gamma _1$ by induction edges that do not belong to the maximal subtree. Let
$u_1,\ldots ,u_n, v_1,\ldots ,v_n\in B$, and compute  normal forms of the conjugacy classes of $u_1,\ldots ,u_n$ and $
v_1,\ldots ,v_n$ with respect to the last HNN-extension, $B=H_{*D}=<B,t|d^t=d', d\in D>.$ Denote these normal forms by
$\bar u_1,\ldots ,\bar u_n$ and $\bar v_1,\ldots ,\bar v_n$. Consider all simultaneous conjugates of normal forms of
$u_1,\ldots ,u_n$, $i=1,\ldots ,k,$ that have the same syllable structure as $\bar v_1,\ldots ,\bar v_n$. We can do
this because the membership problem in maximal abelian subgroups of $B$ is solvable, therefore we can decide when two
elements belong to the same coset of the edge group. If there is no such conjugate, then $u_1,\ldots ,u_n$ and
$v_1,\ldots ,v_n$ are not in the same orbit of $AutC (G)$. Otherwise, make a list of all of them and let us check, one
by one, whether they are in the same $AutC_H(G)$-orbit as $\bar v_1,\ldots ,\bar v_n$, where  $AutC_H(G)$ is the
subgroup of the group of canonical automorphisms $AutC(G)$ fixing $H$. If one does then $u_1,\ldots ,u_n$ and
$v_1,\ldots ,v_n$ are in the same $AutC(G)$-orbit; otherwise, they don't.
\end{proof}

\begin{proposition}
Let ${\mathcal D}$ be an abelian  JSJ decomposition of a freely indecomposable $G\in{\mathcal G}$, with graph $\Gamma$,
and let $H$ be a designated vertex group in $\mathcal D$. Then the WhP  is solvable for the group $AutC_H(G)$ of
canonical automorphisms fixing $H$.
\end{proposition}

\begin{proof}
We use induction on the number of abelian vertex groups and the fact that no two abelian vertices are adjacent to each
other (therefore we can transform the decomposition in such a way that every non-cyclic abelian subgroup is only
connected to one non-abelian vertex group). The base of induction, namely the case when there is no abelian vertex
groups follows from lemma~\ref{base}.

Suppose we can solve the WhP for $AutC(P)$ when $P$ has smaller  number of abelian subgroups. We fix an abelian
subgroup $A$. It is connected only to non-abelian vertex groups in $\mathcal D$, and let us distinguish two cases:

\emph{Case 1: $A$ is connected to only one non-abelian vertex group in $\mathcal D$.} Let $u_1,\ldots ,u_n, v_1,\ldots
,v_n\in G$. We compute normal forms of $v_1,\ldots ,v_n$ (with respect to the amalgamated product $P*_CA$). Denote them
by $\bar v_1,\ldots ,\bar v_n$. Consider all simultaneous conjugates of normal forms of $u_1,\ldots ,u_n$. that have
the same syllable structure as $\bar v_1,\ldots ,\bar v_n$. If there is no such conjugate, then $(u_1,\ldots ,u_n)$ and
$(v_1,\ldots ,v_n)$ are not in the same orbit of $AutC (G)$. Otherwise, make a list of all of them and let us check,
one by one, whether they are in the same $AutC_P(G)$-orbit as $(\bar v_1,\ldots ,\bar v_n)$, where $AutC_P(G)$ is the
subgroup of the group of canonical automorphisms $AutC(G)$ fixing $P$. If one of the conjugates of $(u_1,\ldots ,u_n)$
is in the same $AutC_P(G)$-orbit as $(\bar v_1,\ldots ,\bar v_n)$, then $(u_1,\ldots ,u_n)$ and $(v_1,\ldots ,v_n)$ are
in the same $AutC(G)$-orbit; otherwise, they are not.

To check whether the tuple $(u_1,\ldots ,u_n)$ is in the same $AutC_P(G)$-orbit as a given tuple $(v_1,\ldots , v_n)$
with the same syllable structure, we represent for each $i$ the elements $u_i,v_i$ in normal form as
 $$
u_i=a_0r_1a_1r_2\cdots a_{n-1}r_na_n,\quad v_i=\bar a_0\bar r_1\bar a_1\bar r_2\cdots \bar a_{n-1}\bar r_n\bar a_n.
 $$
By induction, we can check whether there exists an automorphism sending $r_j, j=1,\ldots ,n$, to elements of the form
$c_{1j}\bar r_jc_{2j}$, where $c_{1j},c_{2j}\in C$. If it does not exist, then there is no automorphism sending $u_i$
to $v_i$. If it exists then by Lemma~\ref{base} and Lemma 2.5 from~\cite{LV} there is only a finite number of possible
images for the $r_i$'s. We can apply such automorphism $\alpha$ and assume that
 $$
u_i=a_0r_1a_1r_2\cdots a_{n-1}r_na_n,\quad v_i=\hat a_0r_1\hat a_1r_2\cdots \hat a_{n-1}r_n\hat a_n.
 $$
It only remains to check whether we can extend $\alpha$ in such a way that $\hat a_0=\alpha (a_0), \hat a_i=\alpha
(a_i), \hat a_n=\alpha (a_n)$. If such extension doesn't exist, then $u_i$ and $v_i$ are not in the same orbit. The
argument with tuples is similar.

\emph{Case 2: $A$ is connected to several QH-subgroups.} We represent $A$ as $A=A_1\times A_2$, where the subgroup
generated in $A$ by the edge groups has finite index in $A_1$. Canonical automorphisms map $A_1$ identically to itself
modulo conjugation. We first add $A_1$ to $P$, denote the fundamental group of the obtained graph of groups by $P_1$,
and prove that for any edge group $K$ of $\Gamma$ the problem $SWhP_1(K,A_1)$ is solvable for tuples. Then we consider
$P_1*_{A_1}A$ and repeat the argument done in the first case.

The proposition is proved. This also completes the proof of the theorem.
\end{proof}

\section*{Acknowledgements}
The authors would like to thank the hospitality of CRM-Montr\'{e}al during their research stay there on the fall of 2010.
The first named author gratefully acknowledges partial support from the CUNY and NSERC grants and from MEC (Spain) through
project number MTM2008-01550. The second named author gratefully acknowledges partial support from the MEC (Spain) through
projects number MTM2008-01550 and PR2010-0321.


\begin{thebibliography}{99}

\bibitem{BV}  O. Bogopolski, E. Ventura, On endomorphisms of torsion-free hyperbolic groups, International Journal of
Algebra and Computation, 21 (8) (2011), 1415-1446.
\bibitem{CZ1} D.J. Collins, H. Zieschang, Rescuing the Whitehead method for free products, I. Peak reduction,
Mathematische Zeitschift 185 (4) (1984), 487-504.
\bibitem{CZ2} D.J. Collins, H. Zieschang,  Rescuing the Whitehead method for free products, II. The algorithm,
Mathematische Zeitschift 186 (3) (1984), 335-361.
\bibitem{DG}  F. Dahmani, D. Groves, The Isomorphism Problem for Toral Relatively Hyperbolic Groups.
\bibitem{Dah}  F. Dahmani, Existential questions for (relatively) hyperbolic groups. Israel Journal of Mathematics
Volume 173, Number 1, 91-124,
\bibitem {DGu} F. Dahmani, V. Guirardel, The isomorphism problem for all hyperbolic groups, arXiv:1002.2590v2 [math.GR].
\bibitem{Gro} D. Groves, Limit groups for relatively hyperbolic groups, I. The basic tools. Algebr. Geom. Topol.,
9(3) (2009), 1423-1466.
\bibitem{LV} G. Levitt, K. Vogtmann, A Whitehead algorithm for surface groups, Topology 39 (2000), 1239-1251.
\bibitem{W} J.H.C. Whitehead, On equivalent sets of elements in free groups, Annals of mathematics 37 (1936) 782-800.
\bibitem{GL1} V. Guirardel, G. Levitt, Automorphisms of relatively hyperbolic groups, preprint.
\end{thebibliography}
\end{document}